\documentclass[12pt]{amsart}
\usepackage{amsmath}
\usepackage{graphicx}
\usepackage{hyperref}

\usepackage{mathtools}
\usepackage{amsfonts}
\usepackage{amssymb}
\usepackage[T1]{fontenc}
\usepackage{amsthm}
\usepackage{fullpage}
\usepackage{enumitem}
\usepackage{bbm}
\usepackage[usestackEOL]{stackengine}
\usepackage{multirow}
\newtheorem{theorem}{Theorem}[section]
\newtheorem{corollary}{Corollary}[theorem]
\newtheorem{lemma}[theorem]{Lemma}

\newtheorem{example}{Example}[section]
\newtheorem{proposition}{Proposition}[section]
\theoremstyle{definition}

\theoremstyle{remark}

\numberwithin{equation}{section}

\title{Idempotents of $\mathbb{Z}_n$}


\author{Suman Mondal}
\address{Department of Mathematical Sciences, Tezpur University, Tezpur, Assam, 784028, India}
\email{mondalmondalsuman@gmail.com}

\author{Dhiren Kumar Basnet}
\address{Department of Mathematical Sciences, Tezpur University, Tezpur, Assam, 784028, India}
\email{dbasnet@tezu.ernet.in}

\begin{document}
	\begin{abstract}
		We know that if there are $k$ distinct prime factors of $n \in \mathbb{N}$, then the ring $\mathbb{Z}_n$ of integers modulo $n$ has exactly $2^k$ idempotent elements.  In this article, we try to describe all the idempotents of $\mathbb{Z}_n$ for any given $n \in \mathbb{N}$.
	\end{abstract}
	
	\maketitle
	
	\section{Introduction}
	An element $\overline{a} \in \mathbb{Z}_n$ is said to be an idempotent element if $\overline{a}^2 = \overline{a}$ or $a^2 \equiv a\;(\bmod\; n)$ or $n \mid {a^2-a}$. Clearly $ \overline{0}, \overline{1}$ are  always idempotents of $\mathbb{Z}_n$, for $n (\geq 2) \in \mathbb{N}$. These two idempotents are called the trivial idempotents. If $1 \leq a \leq {n-1}$ and $\overline{a}$ is an idempotent element of $\mathbb{Z}_n$, then trivially $\overline{1} - \overline{a}$ i.e., $\overline{n+1-a}$   is also an idempotent of $\mathbb{Z}_n$. If $n= p_1^{\alpha _1}p_2^{\alpha _2}\cdots p_k^{\alpha _k}$, where $k \in \mathbb{N}$ and for $1\leq i \leq k$, $\alpha_i \in \mathbb{N}$ and $p_i$'s are distinct primes, then there are exactly $2^k$ number of idempotents in $\mathbb{Z}_n$ \cite{mittal}, including $\overline{0}$ and $\overline{1}$.\\
	Let us start with an example and find the $2^3=8$ idempotents of $\mathbb{Z}_{30}$.  We try to find out the $6$ nontrivial idempotents of $\mathbb{Z}_{30}$ and we do that using the definition.\\
	Let $\overline{a}$ be a nontrivial idempotent of $\mathbb{Z}_{30}$, then $30 \mid {a^2-a}$. So, $2\mid a(a-1)$; $3\mid a(a-1)$; $5\mid a(a-1)$, i.e., $2 \mid a$ or $2 \mid a-1$; $3 \mid a$ or $3 \mid a-1$; $5 \mid a$ or $5 \mid a-1$. Out of these choices, we consider the following - \\
	For $2\mid a$, $3 \mid a$, $5 \mid a-1$, we get  $\overline{a}=\overline{6}$ and $\overline{30+1-6}=\overline{25}$ are idempotents of $\mathbb{Z}_{30}$.\\
	For $2\mid a-1$, $3 \mid a$, $5 \mid a $, we get  $\overline{a}=\overline{15}$ and $\overline{30+1-15}=\overline{16}$ are idempotents of $\mathbb{Z}_{30}$.\\
	For $2\mid a$, $3 \mid a-1$, $5 \mid a $, we get $\overline{a}=\overline{25}$ and $\overline{30+1-25}=\overline{6}$ are idempotents of $\mathbb{Z}_{30}$.\\
	For $2\mid a$, $3 \mid a-1$, $5 \mid a  $, we get   $\overline{a}=\overline{10}$ and $\overline{30+1-10}=\overline{21}$ are idempotent of $\mathbb{Z}_{30}$.\\
	Though there are so many choices left to consider, we stop here, as from these choices we can conclde that $\{ \overline{0}, \overline{1}, \overline{6}, \overline{10}, \overline{15}, \overline{16}, \overline{21}, \overline{25} \}$ is the complete list of idempotent elements of $\mathbb{Z}_{30}$.\\
	So, in this example using the definition, easily we can point out the complete list of idempotents of $\mathbb{Z}_{30}$. However, it may be a tedious job to find the idempotents of $\mathbb{Z}_n$ by this trial and error method for larger values of $n$.  \\
	In \cite{sib}, for some particular values of $n$, idempotents of $\mathbb{Z}_{n}$ have been discussed. In this paper, we develop a systematic approach to point out all the idempotents of $\mathbb{Z}_{n}$. From now onwards, we consider that $n(\geq 6) \in \mathbb{Z}$.\\

	\section{preliminaries}
	In this section we develop some preliminary results, that will help us in proving the main results.\\
	\begin{theorem}
		Let $n=2\cdot m$, where $m(\geq 3) \in \mathbb{Z} $ and $(m,2)=1$, then $\overline{m}$ and $\overline{m+1}$ are nontrivial idempotents of $\mathbb{Z}_{n}$.
	\end{theorem}
	\begin{proof}
		Given that $m(\geq 3) \in \mathbb{Z} $ and $(m,2)=1$. Then $m$ must be odd.
		So, $2 \mid m-1$ and trivially $m\mid m(m-1)$. As $(m,2)=1$, we have $n=2\cdot m \mid m^2-m.$
		Therefore, $\overline{m}$ is an idempotent of $\mathbb{Z}_{n}$.\\
		Consequently $\overline{n-m+1} = \overline{m+1}$ is also an idempotent of $\mathbb{Z}_{n}$. As $3 \leq m \leq {n-1}$, so  $\overline{m}$ and $\overline{m+1}$ are nontrivial idempotents of $\mathbb{Z}_{n}$. $\square$
	\end{proof}
	Theorem 2.1 is not always sufficient to point out all idempotents of $\mathbb{Z}_{n}$. However we observe in the next result that, it's possible to find all the idempotents for some particular values of $m$. 
	\begin{corollary}
		Let $n=2\cdot p^\alpha $, where $p$ is an odd prime and $\alpha \in \mathbb{N} $, then $\{ \overline{0}, \overline{1}, \overline{p^\alpha}, \overline{p^\alpha +1}\}$ is the complete list of idempotents of $\mathbb{Z}_n$.
	\end{corollary}
	\begin{proof}
		Taking $m=p^\alpha$ and using Theorem 2.1, we have our desired result.
	\end{proof}
	\begin{example}
		For $n=2\cdot 13^2$, using Corollary 2.1.1, we have that $\{ \overline{0}, \overline{1}, \overline{169}, \overline{170} \} $ is the complete list of idempotents of $\mathbb{Z}_{338}$.
	\end{example}   
	\begin{theorem}
		For $n=3\cdot m$, where $m(\geq 2) \in \mathbb{Z} $ and $(m,3)=1$, we have
		\begin{align*}
			& (a) \; \text{If} \; 3 \mid m+1 \; \text{or} \; 3\mid m-2, \; \text{then } \overline{m+1} \text{ and } \overline{2m} \; \text{ are nontrivial idempotents of } \mathbb{Z}_n,\\
			& (b) \; \text{If} \; 3 \mid m+2 \; \text{or} \; 3\mid m-1, \; \text{then } \overline{m} \text{ and } \overline{2m+1} \; \text{ are nontrivial idempotents of } \mathbb{Z}_n.
		\end{align*}
	\end{theorem}
	

	\begin{theorem}
		For $n=5\cdot m$, where $m(\geq 2) \in \mathbb{Z} $ and $(m,5)=1$, we have
		\begin{align*}
			& (a) \; \text{If} \; 5 \mid m-1 \; \text{or} \; 5\mid m+4, \; \text{then } \overline{m} \text{ and } \overline{4m+1} \; \text{ are nontrivial idempotents of } \mathbb{Z}_n,\\
			& (b) \; \text{If} \; 5 \mid m-2 \; \text{or} \; 5\mid m+3, \; \text{then } \overline{3m} \text{ and } \overline{2m+1} \; \text{ are nontrivial idempotents of } \mathbb{Z}_n,\\
			& (c) \; \text{If} \; 5 \mid m-3 \; \text{or} \; 5\mid m+2, \; \text{then } \overline{2m} \text{ and }  \overline{3m+1} \; \text{ are nontrivial idempotents of } \mathbb{Z}_n,\\
			& (d) \; \text{If} \; 5 \mid m-4 \; \text{or} \; 5\mid m+1, \; \text{then } \overline{4m} \text{ and }  \overline{m+1} \; \text{ are nontrivial idempotents of } \mathbb{Z}_n.
		\end{align*} 
	\end{theorem}
	
	
	Proof and consequences of Theorem 2.2, Theorem 2.3 are similar as of Theorem 2.1.
	\begin{example}
		For $n=165=3\cdot 5\cdot 11$, using Theorem 2.2, out of $8$ idempotents of $\mathbb{Z}_{165}$, we can point out 4 idempotents\; $\overline{0}, \overline{1}, \overline{55}, \overline{111} $.
	\end{example}
	\begin{example}
		For $n=60=2^2\cdot 3\cdot 5$, using Theorem 2.3, out of $8$ idempotents of $\mathbb{Z}_{60}$, we can point out 4 idempotents\; $\overline{0}, \overline{1}, \overline{25}, \overline{36} $.
	\end{example}
	\begin{example}
		For $n=405=5\cdot 9^2$, using Theorem 2.3,  we have that $\{ \overline{0}, \overline{1}, \overline{81}, \overline{325} \} $ is the complete list of idempotents of $\mathbb{Z}_{405}$.
	\end{example}
	We started with the example of $\mathbb{Z}_{30}$ and observed that, using definition, it's not very easy to identify all idempotents of $\mathbb{Z}_{30}$. Previous three theorems help us to point out the complete list of idempotents of $\mathbb{Z}_{30}$,  which is given in the following example.
	\begin{example}
		Let $n=30=2\cdot 3\cdot 5=2 \cdot 15= 3 \cdot 10= 5 \cdot 6$, then for $m=15$, from Theorem 2.1, we obtain that $\overline{15}, \overline{16}$ are idempotents of $\mathbb{Z}_{30}$.\\
		For $m=10$, as $3 \mid 10-1$, from Theorem 2.2, we obtain that $\overline{10}, \overline{21}$ are idempotents of $\mathbb{Z}_{30}$.\\
		For $m=6$, as $5\mid 6-1$, from Theorem 2.3, we obtain that $\overline{6}, \overline{25}$ are idempotents of $\mathbb{Z}_{30}$.\\
		Including the trivial idempotents, we obtain that $ \{ \overline{0}, \overline{1}, \overline{6}, \overline{10}, \overline{15}, \overline{16}, \overline{21}, \overline{25} \} $ is  the complete list of idempotents of $\mathbb{Z}_{30}$.
	\end{example}
	Though these theorems are not sufficient to point out all the idempotents of $\mathbb{Z}_{n}$, for any $n$, these theorems give the preview of the more generalised results we are going to obtain.

	\section{The Main Result}

 In this part, we try to obtain the complete list of idempotents of $\mathbbm{Z}_n$, for any $n(\geq 2) \in \mathbbm{Z}$. We directly don't consider any $n$, rather we start with some special types of $n$ and gradually reach towards our desired result. Let $ n=p \cdot m $ with $p,m>1$ and $(p,m)=1$. Then from \cite{sib}, we obtain that, exacly one  nonzero multiple of $m$(or $p$) is an idempotent of $\mathbbm{Z}_{n}$. Further, there exists some $r$ with  $1\leq r \leq p-1$ such that $rm=qp+1$ for some $q  \in \mathbbm{N} $ and $rm$ is the only multiple of $m$ that is an idempotent of $\mathbbm{Z}_{n}$. So, if it's given that, $a$ is a nontrivial idempotent of $\mathbbm{Z}_{n}$ and $m \mid a$, then for some $q \in \mathbbm{N}$, $a$ is of the form $a=qp+1$. 
\begin{theorem}
    Let $n= p_1^{\alpha _1}p_2^{\alpha _2}\cdots p_k^{\alpha _k} \cdot m$, where $k \in \mathbbm{N}$, $m(\geq 2) \in \mathbbm{N}$ and for $1\leq i \leq k  $, $\alpha_i \in \mathbbm{N} $ and $p_i$'s are distinct primes such that $(p_i,m)=1$, and for $p=p_1^{\alpha _1}p_2^{\alpha _2}\cdots p_k^{\alpha _k}$, we have $m=ps+t$, where $s (\geq 0) \in \mathbbm{Z} $ and $1\leq t \leq p-1 $. Let $1\leq r \leq p-1$, then $\overline{r \cdot m+1}$ and $\overline{(p-r) \cdot m}$ are nontrivial idempotents of $\mathbbm{Z}_{n}$ if and only if $p \mid rt+1$. Such an $r$ is always unique. Further, if $\overline{r \cdot m+1}$ is a nontrivial idempotent of $\mathbbm{Z}_{pm}$and $j=p-t $, then $\overline{(p-r) \cdot m_1+1}$ and $\overline{r \cdot m_1} $ are nontrivial idempotents of $\mathbbm{Z}_{pm_1}$, where $m_1=ps+j$ and $s (\geq 0) \in \mathbbm{Z}$. 
\end{theorem}
\begin{proof}
    With given conditions, let $1\leq r \leq p-1 $  and $p \mid r t+1$, then we prove that $\overline{r \cdot m+1}$, and $\overline{(p-r) \cdot m}$ are nontrivial idempotents of $\mathbbm{Z}_{n}$. 
    Let $\alpha = (rm+1)^2-(rm+1)=rm(rm+1)$, then $m \mid \alpha$.
    Now, given that $m = ps  +t$, i.e.,   $rm+1=prs+rt+1$.
      Since $p \mid rt+1$,  we have $p \mid rm+1$. So $p \mid \alpha $.
    Now, given that $(p_i,m)=1$, for $1\leq i \leq 
 k$ and  $p=p_1^{\alpha _1}p_2^{\alpha _2}\cdots p_k^{\alpha _k}$. So we have $(p,m)=1$.
 Hence, $n=pm \mid (rm+1)^2-(rm+1) $.
 Therefore, $\overline{r \cdot m+1}$ is an idempotent of $\mathbbm{Z}_{n}$ and another idempotent is given by $\overline{n+1-(r \cdot m+1)}= \overline{(p-r) \cdot m} $.
 As $1\leq r, p-r \leq p-1$;  $\overline{r \cdot m+1}$ and $ \overline{(p-r) \cdot m}$ are nontrivial idempotents of $\mathbbm{Z}_{n}$.

 Conversely, let $\overline{r \cdot m+1}$ and $ \overline{(p-r) \cdot m}$ be nontrivial idempotents of $\mathbbm{Z}_{n}$ where $1\leq r \leq p-1$, then we prove that $p \mid rt+1$.
 As $\overline{m \cdot (p-r)}$ (multiple of $m$) is a nontrivial idempotent of $\mathbbm{Z}_{n}$, we have  $m(p-r)=qp+1$, for some $q  \in \mathbbm{N} $. Now $rm+1+m(p-r)=n+1$, i.e., $rm+1+qp+1=n+1$. So $p \mid rm+1$ where $1\leq r \leq p-1 $. 
Let $rt+1 = pq+d$ where $q, d \in \mathbbm{Z}$ such that $0\leq d \leq p-1 $, if $1\leq d \leq p-1 $, then for $t=m-ps $, we have $r(m-pt)+1=pq+d$.
As $p \mid rm+1$, we have $p \mid d$, which is a contradiction.
Hence, we must have $d=0$ i.e., $p \mid r t+1$.

 Now we prove that $r$ is unique.
Let  $1\leq r_1 \leq p-1 $ and $p \mid tr_1+1$ such that $\overline{r_1 \cdot m+1}$ is a nontrivial idempotent of $\mathbbm{Z}_{n}$, then $\overline{(p-r_1)m}$ is also an idempotent and hence from \cite{sib}, we have $\overline{(p-r)m} = \overline{(p-r_1)m}$, which implies $r=r_1$. So, such an $r$ is unique.

Consider that $j=p-t$ and $m_1=ps+j$ where $s (\geq 0) \in \mathbbm{Z}$, then $t=p-j$ and $1\leq j \leq p-1$.
As $\overline{r \cdot m+1}$ is a nontrivial idempotent of $\mathbbm{Z}_{n}$, then by the first part of this theorem, we have $p\mid rt +1$. So, $p\mid r(p-j) +1$, i.e., $p\mid rj-1$. 
Now $rm_1-1=r(ps+j)-1$, so $p \mid rm_1-1$, i.e., for some $u \in \mathbbm{Z}^+$ we have $rm_1-1=pu$.
Hence, $(rm_1)^2-rm_1=rm_1(rm_1-1)=pm_1ru $, implies $pm_1 \mid (rm_1)^2-rm_1$, where $1\leq r \leq p-1$.
So, $\overline{r \cdot m_1}$ is a nontrivial idempotent of $\mathbbm{Z}_{p m_1}$ and the other idempotent is given by
\begin{align*}
 \overline{p \cdot m_1+1-r\cdot m_1}=\overline{(p-r)\cdot m_1+1}.   
\end{align*}
As $1\leq r \leq p-1$, then $\overline{(p-r) \cdot m_1+1}, \overline{r \cdot m_1}$ are nontrivial idempotents of $\mathbbm{Z}_{p m_1}$, where $m_1=ps+j$ and $s (\geq 0) \in \mathbbm{Z}$.     
\end{proof}
In Example 2.4, if we take $n=5^2 \cdot 9^2=2025$ instead of $n=5 \cdot 9^2=405$,  then Theorem 2.3 is not sufficient to point out the idempotents of $\mathbbm{Z}_{2025}$. However, using Theorem 3.1, if we take $p=5^2$ and $m=9^2=5^2\cdot2+31$, then we can obtain the complete list of idempotents of $\mathbbm{Z}_{2025}$. In this case, from Theorem 3.1 we obtain that $t=31$, so we find $r$ such that $1\leq r \leq 24$ and $25 \mid 31r+1$. Here we have $r=4$. Hence, we have that $\{ \overline{0}, \overline{1}, \overline{325}, \overline{1701} \} $ is the complete set of idempotents of $\mathbbm{Z}_{2025}$. Here we can also take $p=9^2$ and $m=5^2$ and proceed in the similar way as before. For different examples, we observe that, Theorem 3.1 is not sufficient to point out all the idempotents of $\mathbbm{Z}_n$ for any $n(\geq 6) \in \mathbbm{Z}$. However, first part of Theorem 3.1 is heavily used later, to obtain more results.\\
In the next theorem, we discuss about implication of the second part of Theorem 3.1.
\begin{theorem}
    Let $p=p_1^{\alpha _1}p_2^{\alpha _2}\cdots p_k^{\alpha _k}$ where $k \in \mathbbm{N}$ and for $1\leq i \leq k  $, $\alpha_i \in \mathbbm{N} $, $p_i$'s are distinct primes. For $1\leq t \leq p-1$ and $s (\geq 0) \in \mathbbm{Z}$, suppose $r$ is unique such that $1\leq r \leq p-1$ and $p \mid rt+1$. If $(p,ps+t)=1$, then  $\overline{r\cdot(ps+t)+1}$ and  $\overline{(p-r) \cdot (ps+t)}$ are nontrivial idempotents of $\mathbbm{Z}_{p (ps+t)}$ if and only if $\overline{(p-r)\cdot(ps+p-t)+1}, \overline{r\cdot (ps+p-t)}$ are nontrivial idempotents of $\mathbbm{Z}_{p (ps+p-t)}$.
\end{theorem}
\begin{proof}
Consider that  $\overline{r\cdot(ps+t)+1}, \overline{(p-r) \cdot (ps+t)}$ are nontrivial idempotents of $\mathbbm{Z}_{p (ps+t)}$.\\
    Let $m=ps+t$, then by hypothesis $(p,m)=1$. For $j=p-t$, let $m_1=ps+j=ps+p-t$, then from the second part of Theorem 3.1, we obtain that $\overline{(p-r)\cdot(ps+p-t)+1}, \overline{r\cdot (ps+p-t)}$ are nontrivial idempotents of $\mathbbm{Z}_{p (ps+p-t)}$.

    Conversely, let $\overline{(p-r)\cdot(ps+p-t)+1}$ and  $\overline{r\cdot (ps+p-t)}$ be nontrivial idempotents of $\mathbbm{Z}_{p (ps+p-t)}$.
    As $(p,ps+t)=1$, we have $(p,ps+p-t)=1$. Let $m_2=ps+p-t$, then $\overline{(p-r)\cdot m_2+1}$ and  $\overline{r\cdot m_2}$ are nontrivial idempotents of $\mathbbm{Z}_{p\cdot m_2}$. As $(p,m_2)=1$, from \cite{sib} we have $rm_2=qp+1$ for some $ q \in \mathbbm{N} $. Now $ (p-r)\cdot m_2+1+ r\cdot m_2 =(p-r)\cdot m_2+1+qp+1 =p\cdot m_2+1 $, i.e., $p \mid 1-rm_2$, which implies $p \mid rt+1$. From Theorem 3.1, it follows that $\overline{r\cdot(ps+t)+1}, \overline{(p-r) \cdot (ps+t)}$ are nontrivial idempotents of $\mathbbm{Z}_{p (ps+t)}$.
\end{proof}
We can explicitly point out some idempotents of $\mathbbm{Z}_{p (ps+p-t)}$ using Theorem 3.1. However, if we know some idempotents of $\mathbbm{Z}_{p (ps+t)}$, then we can point out some more idempotents of $\mathbbm{Z}_{p (ps+p-t)}$, directly from the idempotents of $\mathbbm{Z}_{p (ps+t)}$, without any further calculation. Converse is also true also. These are illustrated in the following example.
\begin{example}
For $n=17 \cdot 40=680$,  taking $p=17$ and $m=40=17\cdot 2+6$ in Theorem 3.1, we get that $\overline{14 \cdot 40+1}=\overline{561}$ and $\overline{(17-14)\cdot 40}=\overline{120}$ are idempotents of $\mathbbm{Z}_{680}$ and we have $r=14$.
Now, using Theorem 3.1, we can explicitly obtain some idempotents of $\mathbbm{Z}_{765}$, however , as $765=17\cdot45=17\cdot \{ 17\cdot 2 +(17-6) \}$, from Theorem 3.2 we get that $\overline{(17-14)\cdot45+1}=\overline{136}$, and $ \overline{14\cdot 45}=\overline{630}$ are  nontrivial idempotents of $\mathbbm{Z}_{765}$.
Similarly, at first if we calculate the idempotent of $\mathbbm{Z}_{765}$ using Theorem 3.1, then using Theorem 3.2 we can obtain some idempotents of $\mathbbm{Z}_{680}$, without using Theorem 3.1.
\end{example}
In Theorem 3.1, we observed that if $\overline{r \cdot m+1}$ is a nontrivial idempotent of $\mathbbm{Z}_{n}$, then the bound for $r$ is given by $1\leq r \leq p-1 $. To calculate $r$, we used the trial and error method on the condition $p\mid rt+1$ for $p-1$ values of $r$. Next we discuss about some condition to make the calculation of $r$ easier and we point out some additional properties of $r$.
\begin{theorem}
    Let $n= p_1^{\alpha _1}p_2^{\alpha _2}\cdots p_k^{\alpha _k} \cdot m$, where $k \in \mathbbm{N}$, $m(\geq 2) \in \mathbbm{N}$ and  $\alpha_i \in \mathbbm{N} $ for $1\leq i \leq k  $, $p_i$'s are distinct primes such that $(p_i,m)=1$ and for $p=p_1^{\alpha _1}p_2^{\alpha _2}\cdots p_k^{\alpha _k}$, we have  $m=ps+t$ where $s (\geq 0) \in \mathbbm{Z} $ and $1\leq t\leq p-1 $. For $1\leq r \leq p-1$ if  $\overline{r \cdot m+1}, \overline{(p-r) \cdot m}$ are nontrivial idempotents of $\mathbbm{Z}_{n}$, then $1\leq \frac{rt+1}{p} \leq t$ and $(p,r)=1$, $(p,rm)=1$.
\end{theorem}
\begin{proof}
    Let $1\leq r \leq p-1$ and $\overline{r \cdot m+1}, \overline{(p-r) \cdot m}$ be nontrivial idempotents of $\mathbbm{Z}_{n}$. Then we have  $t+1 \leq rt+1 \leq t(p-1)+1$. As  $\overline{r \cdot m+1}, \overline{(p-r) \cdot m}$ are nontrivial idempotents of $\mathbbm{Z}_{n}$, from Theorem 3.1, we get that $rt+1=pu$, for some $u\in \mathbbm{Z}^+$. 
    So, $(t+1)/p \leq u \leq \{ t(p-1)+1 \}/p=t+ \{ (1-t)/p \} \leq t$, as  $1\leq t \leq p-1 $.
    Since $(t+1)/p \leq 1$ and $u\in \mathbbm{Z}^+$, we have $1\leq u \leq t$, i.e., $1\leq \frac{rt+1}{p} \leq t$. 
    If $(p,r)=d$, then as $p\mid rt+1$, it follows that $d=1$.
    Also $(p,m)=1$, hence $(p,rm)=1$. 
\end{proof}
Theorem 3.3 provides us a bound for such $u$, discussed in the proof. This $u$ actually makes the calculation of $r$ easier than before. From the fact that $rt+1=pu $, $(p,r)=1$ and $1\leq u \leq t$, now to calculate $r$, we use trial and error method on $r=\frac{pu-1}{t}$, for $t$ number of values of $u$.

For $p=17 $ and $m=60=17 \cdot 3+9$, we have $t=9$ in this case. To calculate the idempotents of $\mathbbm{Z}_{1020}$ of the forms $\overline{r \cdot 60+1}, \overline{(17-r) \cdot 60}$, we first need to calculate $r$ such that $1 \leq r \leq 16$ and $ 17 \mid 9r+1$ (using Theorem 3.1). If we proceed to calculate $r$ using Theorem 3.1, then there are 16 possibilities of $r$. However,  using Theorem 3.3, we observe that 16 possibilities of values of $r$ reduce to $9$ possibilities of values of $u$
 (i.e., to obtain $r$, we need to check only $t$ possible values of $u$ on $rt+1=pu$) and eventually $u=8$, i.e., $r=15$.


Next we define some new sets that contain idempotents of $\mathbbm{Z}_{n}$ and discuss some properties of those sets.

	\section{Some Further Investigation}\label{Further Investigation}
	Let $n= p_1^{\alpha _1}p_2^{\alpha _2}\cdots p_k^{\alpha _k} $, where $k ( \geq 2 )\in \mathbbm{N}$ and for $1\leq i \leq k  $, $\alpha_i \in \mathbbm{N} $ and $p_i$'s are distinct primes. For each $1 \leq j \leq k$, if we write  $n=p_j^{\alpha_j} \cdot m$, where $m={\displaystyle \prod_{i=1, i \neq j}^{k} p_{i}^{\alpha_i}}$, then using Theorem 3.1, we obtain 2 nontrivial idempotents of $\mathbbm{Z}_n$. Now we define some sets, whose elements are idempotents of $\mathbbm{Z}_n$, obtained using Theorem 3.1.\\
 
 Let $1_+ = \{\text{idempotents of }\mathbbm{Z}_n, \text{obtained using Theorem 3.1} : n=p_{j_1}^{\alpha_{{j_1}}} \cdot m, \text{ for some } j_1, 1 \leq j_1\leq k, \alpha_{j_1}\in\mathbbm{N}, m={\displaystyle \prod_{i=1, i \neq j_1}^{k} p_{i}^{\alpha_i}} \}$.\\
  Observe that there are $ {k \choose 1}$ such representations of $n$ and for each such representation we obtain 2 nontrivial idempotent of $\mathbbm{Z}_n$.
 So, $  |1_+| \leq 2 {k \choose 1}$.  Similarly, let \\
 $ 2_+ =$ $\{\text{idempotents of }\mathbbm{Z}_n, \text{obtained using Theorem 3.1} : n=p_{j_1}^{\alpha_{{j_1}}}  p_{j_2}^{\alpha_{{j_2}}}  m, \text{ for some } j_1,j_2, 1 \leq j_1 \neq j_2\leq k, m={\displaystyle \prod_{i=1, i \neq j_1, j_2}^{k} p_{i}^{\alpha_i}} \}$.\\
  Then by the similar argument, $ |2_+| \leq 2 {k \choose 2}$. \\
 In general, for every $h$ with $1 \leq h \leq k-1$, if\\
  $  h_+ =$ $\{\text{idempotents of }\mathbbm{Z}_n, \text{ we obtain using Theorem 3.1} : n=p_{j_1}^{\alpha_{{j_1}}} p_{j_2}^{\alpha_{{j_2}}} \cdots p_{j_h}^{\alpha_{{j_h}}} m, \text{ for some }\\
 j_1,j_2, \cdots, j_h, 1 \leq j_1 \neq j_2 \neq \cdots \neq j_h\leq k, m={\displaystyle \prod_{i=1, i \neq j_1,\cdots, j_h}^{k} p_{i}^{\alpha_i}} \}$.\\
 
Then $  |h_+| \leq 2 {{k \choose h}}$. Next we discuss some results, that help us to prove that $\  |h_+| = 2 {k \choose h}$.
 \begin{lemma}
      Let $n= p_1^{\alpha _1}p_2^{\alpha _2}\cdots p_k^{\alpha _k}$, where $k(\geq 2) \in \mathbbm{N}$ and for $1\leq i \leq k  $, $\alpha_i \in \mathbbm{N} $, $p_i$'s are distinct primes.  Let $n=p m={p'}  {m'}$, where $p=p_{j_1}^{\alpha_{{j_1}}}  p_{j_2}^{\alpha_{{j_2}}} \cdots p_{j_h}^{\alpha_{{j_h}}} $, $m=p_{j_{h+1}}^{\alpha_{{j_{h+1}}}}  p_{j_{h+2}}^{\alpha_{{j_{h+2}}}} \cdots p_{j_k}^{\alpha_{{j_k}}} $, ${p'}=m$, ${m'}=p$, for $\alpha_{j_i} \in \mathbbm{N}$, $1\leq j_i \leq k$, $\forall i=1,2,\cdots ,k$ and $1\leq j_h \leq k-1$. Then the two idempotents, obtained for each of the  different representations $n=p m$ and $n={p'}{m'}$, using Theorem 3.1, are same.
 \end{lemma}
 \begin{proof}
     Let $\overline{r \cdot m+1}, \overline{(p-r) \cdot m}$ be the two nontrivial idempotents obtained using Theorem 3.1 for the representation $n=p m$, where $m=ps+t$, $s (\geq 0) \in \mathbbm{Z} $, $1\leq t \leq p-1 $ and $1\leq r \leq p-1$ with $p \mid rt+1$, and let $\overline{{r'} \cdot {m'}+1}, \overline{({p'}-{r'}) \cdot {m'}}$ be the two idempotents obtained using Theorem 3.1 for the representation $n={p'} {m'}$, where ${m'}={p'}{s'}+{t'}$, ${s'} (\geq 0) \in \mathbbm{Z} $, $1\leq {t'} \leq {p'}-1 $ and $1\leq {r'} \leq {p'}-1$ with ${p'} \mid {r'}{t'}+1$.
     Now given that $n=p\cdot m={p'} \cdot {m'}$ and ${p'}=m$, ${m'}=p$. Then, to obtain the result it's enough to prove that $ \{ A,B \} = \{ C,D \} $, where            $ A =\overline{r \cdot m+1} $, $B=\overline{(p-r) \cdot m}$, $C=\overline{{r'} \cdot {m'}+1}$, $D=\overline{({p'}-{r'}) \cdot {m'}}$.
     We claim $A =D$.
     Assume that $rm+1-{p'}{m'}+{r'}{m'}=\alpha$. As $\overline{r \cdot m+1}$ is an idempotent of $\mathbbm{Z}_{n(=pm)}$, we have $p \mid rm+1$, so $p \mid \alpha$. Similarly, as $\overline{{r'} \cdot {m'}+1}$ is an idempotent of $\mathbbm{Z}_{n=(p'm')}$, we have ${p'}=m \mid \alpha$ with $(p,m)=1$.
     Therefore, $n\mid \alpha=(rm+1)-({p'}-{r'})m'$.
     Hence $\overline{r \cdot m+1}=\overline{({p'}-{r'}) \cdot {m'}}$, i.e., $A =D$.
     Similarly we can prove that $\overline{{r'} \cdot {m'}+1}=\overline{(p-r) \cdot m}$, i.e., $B=C$.
 \end{proof}
 Whether $A$ and $B$ are distinct or not, that will be discussed later but from Lemma 4.1, we observe that, $ \{ A,B \} \subseteq {h_+} $ and $ \{ C,D \} \subseteq {(k-h)_+} $, where $1 \leq h \leq k-1$. So Lemma 4.1 gives us some information about elements of $i_+$ and $j_+$, where $i+j=k$ and $1\leq i \neq j \leq k-1$. 
 \begin{lemma}
     Let $n= p_1^{\alpha _1}p_2^{\alpha _2}\cdots p_k^{\alpha _k}\cdot m $, where $k \in \mathbbm{N}$, $1\leq i \leq k  $, $\alpha_i \in \mathbbm{N} $, $p_i$'s are distinct primes and $(p_i,m)=1$. If $p=p_1^{\alpha _1}p_2^{\alpha _2}\cdots p_k^{\alpha _k}$, then there exists $r$ with $1\leq r \leq p-1$ such that  $\overline{r \cdot m+1}$ and $\overline{(p-r) \cdot m}$ are distinct nontrivial idempotents of $\mathbbm{Z}_{n}$. 
 \end{lemma}
 \begin{proof}
     Let $(x_1,-y_1)$ be a solution or Bezout's coefficient of $px+my=1$, where $x_1,y_1$ are positive integers. Then $p\mid my_1+1$.
     Let $r=y_1$, then trivially $r\geq 1$ and $p\mid rm+1$ and   $r=y_1={|-y_1|}\leq |\frac{p}{(p,m)}|=p$. If $r=y_1=p$, then $p=1$, which is not possible.
     Hence $1\leq r \leq p-1$.
     As $(p,m)=1$ and $p\mid rm+1$, for $m=ps+t$ we have $p\mid rt+1$, where $1\leq t \leq p-1$ and $s (\geq 0)\in \mathbbm{Z}$. 
     From Theorem 3.1, it follows that $\overline{r \cdot m+1}$ and $\overline{(p-r) \cdot m}$ are distinct nontrivial idempotents of $\mathbbm{Z}_{n}$, where $1\leq r \leq p-1$.
 \end{proof}
 For given $n(\geq 6) \in \mathbbm{Z}$, Lemma 4.2 ensures that there exists some $r$ with $1\leq r \leq p-1$ such that $\overline{r \cdot m+1}$ and $\overline{(p-r) \cdot m}$ are distinct nontrivial idempotents of $\mathbbm{Z}_{n}$. 
 \begin{theorem}
    Let $n= p_1^{\alpha _1}p_2^{\alpha _2}\cdots p_k^{\alpha _k}$, where $k(\geq 2) \in \mathbbm{N}$, $1\leq i \leq k  $, $\alpha_i \in \mathbbm{N} $, $p_i$'s are distinct primes, then for $1 \leq h \leq k-1$, the following hold. 
    \begin{enumerate}[label=(\alph*)]
        \item $| h_+|=2 {k \choose h}$,\text{ if}  $h\neq k/2$,\\ 
        \item $| h_+|= {k \choose h}$,\text{ if}  $h=k/2$.
    \end{enumerate}
 \end{theorem}
 \begin{proof}
     Recall that  for  $h$ with $1 \leq h \leq k-1$, \\
     $  h_+=$ $\{\text{idempotents of }\mathbbm{Z}_n, \text{  obtained using Theorem 3.1} : n=p_{j_1}^{\alpha_{{j_1}}}  p_{j_2}^{\alpha_{{j_2}}} \cdots p_{j_h}^{\alpha_{{j_h}}} m, \text{ for some }\\j_1,j_2, \cdots, j_h, 1 \leq j_1 \neq j_2 \neq \cdots \neq j_h\leq k, m={\displaystyle \prod_{i=1, i \neq j_1,\cdots, j_h}^{k} p_{i}^{\alpha_i}} \}$.\\
     Now there are ${k} \choose {h}$ different representations of $n$ of the form described in $h_+$ and for each representation, we obtain two nontrivial idempotents of $\mathbbm{Z}_n$ using Theorem 3.1.
     To prove the theorem, it's enough if we choose two arbitrary representations of $n$ of the form discussed in $h_+$ and show that the four idempotent elements obtained using Theorem 3.1 are distinct.\\
     Let $p=p_{j_1}^{\alpha_{{j_1}}} p_{j_2}^{\alpha_{{j_2}}} \cdots p_{j_h}^{\alpha_{{j_h}}} $ and $m=p_{j_{h+1}}^{\alpha_{{j_{h+1}}}} p_{j_{h+2}}^{\alpha_{{j_{h+2}}}} \cdots p_{j_k}^{\alpha_{{j_k}}} $, where $1 \leq j_i \leq k$, $\alpha_{j_i} \in \mathbbm{N}$ for $i= 1,2, \cdots ,k$ and $1 \leq j_h \leq k-1$ for some $h$ such that $1 \leq h \leq k-1$. Then for $n=p\cdot m$ and $m=ps+t$, where $s(\geq 0) \in \mathbbm{Z}$ and $1 \leq t,r \leq p-1$ , we have that $\overline{r \cdot m+1}$ and $ \overline{(p-r) \cdot m}$ are nontrivial idempotent of $\mathbbm{Z}_{n}$ such that  $ p\mid rt+1$ ( using Theorem 3.1, Lemma 4.2).\\
     Next let ${p'}=p_{g_1}^{\alpha_{{g_1}}}  p_{g_2}^{\alpha_{{g_2}}} \cdots p_{g_h}^{\alpha_{{g_h}}} $ and ${m'}=p_{g_{h+1}}^{\alpha_{{g_{h+1}}}} p_{g_{h+2}}^{\alpha_{{g_{h+2}}}} \cdots p_{g_k}^{\alpha_{{g_k}}} $ where $1 \leq g_i \leq k$, $\alpha_{g_i} \in \mathbbm{N}$ for $i= 1,2, \cdots ,k$ and $1 \leq g_h \leq k-1$ for some $h$ such that $1 \leq h \leq k-1$. Note that, here $p_{g_e}^{\alpha_{{g_e}}}$'s are nothing but $p_{j_f}^{\alpha_{{j_f}}}$'s, but possibly in a different order. Then for $n={p'}\cdot {m'}$ and ${m'}={p'}{s'}+{t'}$, where ${s'}(\geq 0) \in \mathbbm{Z}$ and $1 \leq {t'}, {r'} \leq {p'}-1$, we have that $\overline{{r'} \cdot {m'}+1}, \overline{({p'}-{r'}) \cdot {m'}}$ are nontrivial idempotents of $\mathbbm{Z}_{n}$ such that  $ {p'}\mid {r'}{t'}+1$ (using Theorem 3.1, Lemma 4.2).
     Now $n= pm = {p'} {m'}$. If we take $p={p'}$, then it implies that $m={m'}$ and we obtain 2 nontrivial idempotents (using Lemma 4.1) of $\mathbbm{Z}_n$. In that case, we are not able to get any relation among the idempotents of $\mathbbm{Z}_n$, obtained using Theorem 3.1.
     So, we take $p\neq{p'}$ and $m\neq{m'}$.
     If we consider $k$ is even and $k=2h$, then for $p=p_{j_1}^{\alpha_{{j_1}}} p_{j_2}^{\alpha_{{j_2}}} \cdots p_{j_h}^{\alpha_{{j_h}}} $, $m=p_{j_{h+1}}^{\alpha_{{j_{h+1}}}} p_{j_{h+2}}^{\alpha_{{j_{h+2}}}} \cdots p_{j_k}^{\alpha_{{j_k}}} $,  and
     ${p'}=m=p_{j_{h+1}}^{\alpha_{{j_{h+1}}}} p_{j_{h+2}}^{\alpha_{{j_{h+2}}}} \cdots p_{j_k}^{\alpha_{{j_k}}}$, ${m'}=p=p_{j_1}^{\alpha_{{j_1}}}  p_{j_2}^{\alpha_{{j_2}}} \cdots p_{j_h}^{\alpha_{{j_h}}} $, we have $n= pm = {p'} {m'}$ and for these 2 different representations of $n$ we obtain 2 nontrivial idempotents (using Lemma 4.1) of $\mathbbm{Z}_n$.
     So, we exclude the condition of $k=2h$ and we discuss this in the second part. Here we take $p \neq {m'}$ and $m\neq {p'}$.
     Hence, we must have $p \neq {m'}$, $m\neq {p'}$, $p\neq{p'}$ and $m\neq{m'}$, i.e., $p, m, {p'}, {m'}$ are all distinct.
     Let  $ \alpha=\overline{r \cdot m+1} $, $\beta=\overline{(p-r) \cdot m}$, $\gamma=\overline{{r'} \cdot {m'}+1}$, $\delta=\overline{({p'}-{r'}) \cdot {m'}}$ be the idempotents, then to obtain the result in the first part, it's enough to show that $\alpha, \beta, \gamma, \delta$ are all distinct.\\
     If possible, let $\alpha =\beta$. Then we have $\overline{r \cdot m+1}= \overline{(p-r) \cdot m}$, i.e., $n \mid rm+1-pm+rm$.
     As $n=pm$ and $\overline{r \cdot m+1}$ is a nontrivial idempotent of $\mathbbm{Z}_n$ then $p \mid rm+1$.
     Therefore, $p\mid rm$, where $(p,m)=1$. Then $p \mid r$, where $1 \leq r \leq p-1$, which is a contradiction.
     Hence, we have $\alpha \neq \beta$ and similarly we can prove that $\gamma \neq \delta$.\\
     Next we prove that $\alpha \neq \gamma$.
     If possible, let $\overline{r \cdot m+1} = \overline{{r'}\cdot {m'} +1}$ where $1<rm+1,{r'} {m'} +1 <n $. Then we have $rm+1={r'} {m'} +1$, i.e., $rm={r'} {m'}$.
     Now we have $pm={p'} {m'}$, which implies $pmr={p'} {m'}r$. As $rm={r'} {m'}$, then we have  $p{r'} {m'}=nr$, i.e., $n={p'} {m'} \mid p{r'} {m'}$, where $({p'}, {r'})=1$ (from Theorem 3.3).
     So, we have ${p'} \mid p$.\\ Similarly, for $pm{r'}={p'} {m'}{r'}$ we have $n=pm \mid {p'}mr$, where $(p,r)=1$ (from Theorem 3.3).\\ So, we have $p \mid {p'}$.
     Therefore, $p= {p'}$, which is a contradiction as $p \neq {p'}$.
     Hence, we have $rm+1\neq {r'} {m'} +1$, where $1<rm+1,{r'} {m'} +1 <n $ and consequently $\overline{r \cdot m+1} \neq \overline{{r'}\cdot {m'} +1}$, i.e., $\alpha \neq \gamma$. 
     We have $\overline{r \cdot m+1} \neq \overline{{r'}\cdot {m'} +1}$ and $n= pm = {p'} {m'}$. So, we get that $\overline{r \cdot m} \neq \overline{{r'}\cdot {m'} }$, which implies $\overline{n-r \cdot m} \neq \overline{n-{r'}\cdot {m'} }$, i.e., $\overline{(p-r) \cdot m} \neq \overline{({p'}-{r'}) \cdot {m'}}$.
     Therefore, $\beta \neq \delta$.\\
     Next we prove that $\overline{r \cdot m+1} \neq \overline{({p'}-{r'}) \cdot {m'}} $.
     If possible, let $\overline{r \cdot m+1} = \overline{({p'}-{r'}) \cdot {m'}} $.
     As  $2 < {r \cdot m+1},{({p'}-{r'}) \cdot {m'}} <n$, then we get that $r  m+1=({p'}-{r'})  {m'}$. We know that $n=pm={p'} {m'}$ and $p \mid rm+1$, ${p'} \mid {r'}  {m'}+1 $, which implies $p \mid {r'}{m'}$ and ${p'} \mid rm$.\\
     Let $rm = {p'}u$, where $u \in \mathbbm{Z}^+$. Then we get that $rm{m'}= {p'} {m'}u= nu$, i.e., $n=pm \mid rm{m'}$. As $(p,r)=1$ (using Theorem 3.3), then we get that $p \mid {m'}$. Similarly we get that $ {p'} \mid m$.\\
     For $a, b \in \mathbbm{Z}^+$, let $m={p'}a$ and ${m'} =pb$, then $n=pm={p'} {m'}$ implies that $a =b$. So we get that $m={p'}a$ and ${m'} =pa$.\\
     Now as we considered $\overline{r \cdot m+1} = \overline{({p'}-{r'}) \cdot {m'}}$, we get that $m \mid {r'}  {m'}+1 $. This implies that $(m,{m'})=1$. We have $a$ as the common divisor of $m$ and ${m'}$, then $a$ must be equal to 1. This implies that $m={p'}$ and ${m'} =p$, which is a contradiction as $m \neq {p'}$ and ${m'}  \neq p$.\\
     Hence, $r  m+1\neq ({p'}-{r'})  {m'}$, where $2 < {r \cdot m+1},{({p'}-{r'}) \cdot {m'}} <n$. Then we get that $\overline{r \cdot m+1} \neq \overline{({p'}-{r'}) \cdot {m'}} $.\\
     Therefore, $\alpha\neq \delta$. Similarly, we can prove that $\beta \neq \gamma$.\\
     Now from the six possible relations among $\alpha,\beta,\gamma,\delta$, we conclude that $\alpha,\beta,\gamma,\delta$ are all distinct. If $h \neq k/2$, then for every representation of $n$ in $h_+$, we obtain 2 unique nontrivial idempotent of $\mathbbm{Z}_n$. In $h_+$, there are ${k} \choose {h}$ such representations of $n$ and for each representation we obtain 2 unique nontrivial idempotent of $\mathbbm{Z}_n$. Hence, $| h_+|=2 {k \choose h}$,  where  $h\neq k/2$.
    
    For the second part, let $k=2h$, then there are ${k \choose h}$ distinct representations of $n= p_1^{\alpha _1}p_2^{\alpha _2}\cdots p_k^{\alpha _k}$ and out of these representations, we can find pairs of distinct representations of the form $(n=pm,n={p'} {m'} )$ such that $p=p_{j_1}^{\alpha_{{j_1}}} \cdot p_{j_2}^{\alpha_{{j_2}}} \cdots p_{j_h}^{\alpha_{{j_h}}} $, $m=p_{j_{h+1}}^{\alpha_{{j_{h+1}}}} \cdot p_{j_{h+2}}^{\alpha_{{j_{h+2}}}} \cdots p_{j_k}^{\alpha_{{j_k}}} $, ${p'}=m$, ${m'}=p$, for $\alpha_{j_i} \in \mathbbm{N}$, $1\leq j_i \leq k$, $\forall i=1,2,\cdots ,k$ and  $1\leq j_h \leq k-1$ with $1 \leq i \neq g \leq k$ for $j_i \neq j_g$.
     From Lemma 4.1, for these 2 different representations of $n$, we have 2 distinct nontrivial idempotents of $\mathbbm{Z}_n$, i.e., in this case, each such pair in $h_+$ gives 2 distinct nontrivial idempotents of $\mathbbm{Z}_n$.
     Now there are $ \frac{1}{2}  {k \choose h}$ such pairs in $h_+$. As $k $ is even,  $ \frac{1}{2}  {k \choose h}$ is eventually a natural number. In total, in $h_+$ we get ${k \choose h}$ number of idempotents in   $\mathbbm{Z}_n$.
     Using the same process as in the first part of this theorem, we can show that, the ${k \choose h}$ idempotents in $h_+$ are all distinct.
     Therefore, $|{h_+|}={k \choose h}$, where $k$ is even and $k=2h$. This completes the proof.
 \end{proof}
 Next we discuss about the relation of $g_+$ and $h_+$, where $1\leq g \neq h <k$. These results lead to the main part of obtaining all idempotents of $\mathbbm{Z}_n$. 
 \begin{corollary}
     For $n= p_1^{\alpha _1}p_2^{\alpha _2}\cdots p_k^{\alpha _k}$, where $k(\geq 2) \in \mathbbm{N}$, $1\leq i \leq k  $, $\alpha_i \in \mathbbm{N} $, $p_i$'s are distinct primes, if $g+h=k$ where $1\leq g \neq h <k$, then $g_+=h_+$. Otherwise, $g_+ \cap h_+ = \phi$ whenever $g_+ \neq h_+$ and $g+h \neq k$, for $1\leq g \neq h <k$.
 \end{corollary}
 \begin{proof}
     For given $n= p_1^{\alpha _1}p_2^{\alpha _2}\cdots p_k^{\alpha _k}$, where $k(\geq 2) \in \mathbbm{N}$, $1\leq i \leq k  $, $\alpha_i \in \mathbbm{N} $, $p_i$'s are distinct primes, let $g+h=k$, where $1\leq g \neq h <k$. Then $g \neq k/2 $, $h \neq k/2$, in case $k$ is even.\\
     So, from Theorem 4.1, we get that $ | g_+|=2{k \choose g}$ and $|h_+|=2{k \choose h}$.
     As $g+h=k$, then we have ${k \choose h}={k \choose k-h}={k \choose g}$.
     Therefore, $|g_+|=|h_+|$.\\
     To prove that $g_+=h_+$, it's enough to show that, in $g_+$ for a representation of the form $n=pm$, there is an unique representation of the form $n= {p'} {m'}$ in $h_+$ such that the nontrivial idempotents of $\mathbbm{Z}_n$, we get for these two different representations of $n$, are same.\\
     Let $p=p_{j_1}^{\alpha_{{j_1}}} p_{j_2}^{\alpha_{{j_2}}} \cdots p_{j_g}^{\alpha_{{j_g}}} $, $m=p_{j_{g+1}}^{\alpha_{{j_{g+1}}}} p_{j_{g+2}}^{\alpha_{{j_{g+2}}}} \cdots p_{j_k}^{\alpha_{{j_k}}} $, ${p'}=m$, ${m'}=p$, for $\alpha_{j_i} \in \mathbbm{N}$,$1\leq j_i \leq k$, $\forall i=1,2,\cdots ,k$ and $1\leq j_g \leq k-1$. Then the two idempotents we obtain for each representation $n=pm$ and $n= {p'} {m'}$, using Theorem 3.1, are same in some order, i.e., if $\overline{r \cdot m+1}$ and $\overline{(p-r) \cdot m}$ (elements of $g_+$) be idempotents of $\mathbbm{Z}_n$ for $n=pm$, and $\overline{{r'} \cdot {m'}+1}, \overline{({p'}-{r'}) \cdot {m'}}$ (elements of $h_+$) be idempotents of $\mathbbm{Z}_n$ for $n= {p'} {m'}$, then from Lemma 4.2  we get $\overline{r m+1}=\overline{({p'}-{r'}) {m'}}$ and $\overline{{r'} \cdot {m'}+1}=\overline{(p-r) \cdot m}$, where $r, {r'}$ are as discussed in Theorem 3.1.
     For the choice of $p,m$ with $n=pm$, the choice of ${p'}=m$ and ${m'}=p$ is unique.
      Hence, $g_+=h_+$.  

      For the second part, let $g_+\neq h_+$, where $1 \leq g \neq h < k$ with $g+h \neq k$, and let $p=p_{j_1}^{\alpha_{{j_1}}}  p_{j_2}^{\alpha_{{j_2}}} \cdots p_{j_g}^{\alpha_{{j_g}}} $, $m=p_{j_{g+1}}^{\alpha_{{j_{g+1}}}} p_{j_{g+2}}^{\alpha_{{j_{g+2}}}} \cdots p_{j_k}^{\alpha_{{j_k}}} $, ${p'}=p_{u_1}^{\alpha_{{u_1}}}  p_{u_2}^{\alpha_{{u_2}}} \cdots p_{u_h}^{\alpha_{{u_h}}} $, ${m'}=p_{u_{h+1}}^{\alpha_{{u_{h+1}}}} p_{u_{h+2}}^{\alpha_{{u_{h+2}}}} \cdots p_{u_k}^{\alpha_{{u_k}}} $, where $1 \leq j_i, u_i \leq k$ and  $\alpha_{j_i}, \alpha_{u_i} \in \mathbbm{N}$ for $i= 1,2, \cdots ,k$; and $1 \leq j_h, u_h \leq k-1$ for some $h$ such that $1 \leq h \leq k-1$, and $j_A \neq j_B, u_A \neq u_B$ for $1 \leq A \neq B \leq k$. Then we get that $p \neq {p'}, m \neq {m'}, {p'} \neq m, {m'} \neq p $.\\
      Similar to the first part of Theorem 4.1, we  prove that the two idempotent of $\mathbbm{Z}_n$ for $n=pm$, and the two idempotent of $\mathbbm{Z}_n$ for $n= {p'}{m'}$ (in total these 4 idempotents) are distinct idempotents of $\mathbbm{Z}_n$.
      Therefore, the result is true for any representation of $n$ in $g_+, h_+$.\\
      So, for $g_+\neq h_+$  and $g+h \neq k$, with $1\leq g \neq h <k$, we have $g_+ \cap h_+ = \phi$.
  \end{proof}

	\section{Final Results}\label{Final Results}
	In this part, we discuss multiple proposition which lead to the methods of calculating all distinct idempotents of $\mathbbm{Z}_n$. Most importantly, we observe that, every nontrivial idempotent of $\mathbbm{Z}_n$ has some expression of specific form. 
 \begin{proposition}
 For $n= p_1^{\alpha _1}p_2^{\alpha _2}\cdots p_k^{\alpha _k}$, where $k(\geq 2) \in \mathbbm{N}$, $1\leq i \leq k  $, $\alpha_i \in \mathbbm{N} $, $p_i$'s are distinct primes, the following hold.
 \begin{enumerate}[label=(\alph*)]
       \item If $k$ is odd, then the $2^k$ idempotents of $\mathbbm{Z}_n$ are given by the elements of  ${1_+},{2_+}, \cdots $, ${(\frac{k-1}{2})_+}$.
       \item If $k$ is even, then the $2^k$ idempotents of $\mathbbm{Z}_n$ are given by the elements of ${1_+},{2_+}, \cdots $, $ {(\frac{k}{2}-1)_+} $, and $(\frac{k}{2})_+$. We calculate the elements of $(\frac{k}{2})_+$ in such a way that for any 2 distinct representations of $n$ of the form $ n=pm$, $n={p'} {m'}$ in $(\frac{k}{2})_+$, we have $p \neq {m'}, {p'} \neq m$.
   \end{enumerate}
   \end{proposition}
   \begin{proof}
       Let $k$ be odd, then we can get the elements of ${1_+},{2_+}, \cdots , {(\frac{k-1}{2})_+},{(\frac{k+1}{2})_+}, \cdots,(k-2)_+, (k-1)_+$. We have $1+(k-1)=2+(k-2)= \cdots = (\frac{k-1}{2})+(\frac{k+1}{2})=k$.\\
       From Corollary 4.3.1, we get that, $1_+=(k-1)_+, 2_+=(k-2)_+, \cdots, (\frac{k-1}{2})_+=(\frac{k+1}{2})_+$ and $g_+ \cap h_+= \phi$ where $1 \leq g \neq h \leq (\frac{k-1}{2}) < k$. We know that ${{2n+1} \choose {0}}+{{2n+1} \choose {1}}+ \cdots +{{2n+1} \choose {n}}= 2^{2n}$.\\
       Now from Theorem 4.1, we get that, $|{h_+|}=2{ {k} \choose {h}}$ where $1 \leq h \leq k-1.$ Thus the number of nontrivial distinct idempotents of $\mathbbm{Z}_n$, in this case,  is $2{{k} \choose {1}}+2{{k} \choose {2}}+ \cdots +2{{k} \choose {\frac{k-1}{2}}}= 2 \cdot (2^  {\frac{2(k-1)}{2}} -1)=2^k-2$.
       Including the trivial idempotents,  the elements of  ${1_+},{2_+},\cdots,$ $ {(\frac{k-1}{2})_+}$, yield the $2^k$ idempotents of $\mathbbm{Z}_n$.

       For the second part, let $k$ be even, then we can have the elements of  ${1_+},{2_+}, \cdots , {(\frac{k}{2}-1)_+}, (\frac{k}{2})_+,$ ${(\frac{k}{2}+1)_+}, \cdots,(k-2)_+, (k-1)_+$. We have $1+(k-1)=2+(k-2)= \cdots = (\frac{k}{2}-1)+(\frac{k}{2}+1)=k$.
       From Corollary 4.3.1, we get that, $1_+=(k-1)_+, 2_+=(k-2)_+, \cdots, (\frac{k}{2}-1)_+=(\frac{k}{2}+1)_+$ and $g_+ \cap h_+= \phi$ where $1 \leq g \neq h \leq (\frac{k}{2}-1) < k$. Additionally we get the elements of $(\frac{k}{2})_+$ and in $(\frac{k}{2})_+$, for two different representations $n=pm$ and $n={p'} {m'}$, if we consider that $p={m'}$ and  ${p'}=m$, then using Lemma 4.1, for these two different representations of $n $, we obtain the same idempotents.
       So, we get the elements of  $(\frac{k}{2})_+$ in such a way that for any 2 distinct representations of $n$ of the form $ n=pm$, $n={p'} {m'}$ in $(\frac{k}{2})_+$, we have $p \neq {m'},$ and $ {p'} \neq m$. In that case, from Theorem 4.1, we have $|  (\frac{k}{2})_+|= { {k} \choose {\frac{k}{2}}}$.\\
       Here the number of nontrivial distinct idempotents of $\mathbbm{Z}_n$ is 
       \begin{align*}
           & = 2{{k} \choose {1}}+2{{k} \choose {2}}+ \cdots +2{{k} \choose {\frac{k}{2}-1}}+ {{k} \choose {\frac{k}{2}}}\\
           & = {{k} \choose {1}}+{{k} \choose {k-1}}+{{k} \choose {2}}+{{k} \choose {k-2}}+ \cdots + {{k} \choose {\frac{k}{2}-1}}+{{k} \choose {\frac{k}{2}+1}}+{ {k} \choose {\frac{k}{2}}}\\
           & = {{k} \choose {1}}+{{k} \choose {2}}+ \cdots +{{k} \choose {k-1}}=  2^k- 2.          
       \end{align*}
       Including the trivial idempotents, the elements of ${1_+},{2_+}, \cdots , {(\frac{k}{2}-1)_+},$ $(\frac{k}{2})_+$ with some conditions on $(\frac{k}{2})_+$, give us the  $2^k$ idempotents of $\mathbbm{Z}_n$.
   \end{proof}
   Using  Proposition 5.1, next we present the working method to calculate all the distinct idempotents of $\mathbbm{Z}_n$.\\
\noindent\underline{\textbf{Working method (1)(a)}} : Consider that $n=2^2\cdot3^2\cdot5\cdot7\cdot11=13860$. Here $k=5$. Now we calculate the elements of $1_+$, $2_+$ in the  Table 5.1 below, using the first part of Proposition 5.1.
   \begin{center}
   \strutlongstacks{T}
\begin{tabular}{ |c|c|c|c|c|c|c| } 
\hline
$\Longstack{set\\of\\idempo-\\tents}$ & \Longstack{$p= \displaystyle \prod_{i=1} ^{h}p_{i}^{\alpha_i}$\\ \\$1 \leq h \leq 2$} & \Longstack{$m={\displaystyle \prod_{i=1}^{k} p_{i}^{\alpha_i}}$\\ \\$i \neq j_1,\cdots, j_h$\\$(m,p)=1$} &\Longstack{$m=ps+t,$ \\$s (\geq 0)\in \mathbbm{Z}$\\ $1 \leq t \leq p-1$} & \Longstack{ $1 \leq r \leq p-1, $ \\ $p \mid rt+1$ \\$(p,r)=1$ }& {idempotents}&idempotents\\
\hline
$h_+$ & $p$ & $m$ & $t$ & $r$ & $\overline{r \cdot m+1}$ & $\overline{(p-r)\cdot m}$\\
\hline
\multirow{5}{1em}{$1_+$} & $2^2$ & $3^2 \cdot 5 \cdot 7 \cdot 11 $& 1 & 3 & $\overline{10396}$ & $\overline{3465}$\\ 
& $3^2$ & $2^2 \cdot 5 \cdot 7\cdot 11 $ & 1 & 8 &$\overline{12321}$ & $\overline{1540}$\\ 
& 5 & $2^2 \cdot 3^2 \cdot 7 \cdot 11 $& 2 & 2 & $\overline{5545}$ & $\overline{8316}$ \\ 
& 7 & $2^2 \cdot 3^2 \cdot 5 \cdot 11 $& 6 & 1 & $\overline{1981}$ & $\overline{11880}$\\
& 11 & $2^2 \cdot 3^2 \cdot 5 \cdot 7 $& 6 & 9 & $\overline{11341}$ & $\overline{2520}$\\
\hline
\multirow{10}{1em}{$2_+$} & $2^2 \cdot 3^2$ & $ 5 \cdot 7 \cdot 11 $& 25 & 23 & $\overline{8856}$ & $\overline{5005}$\\
& $2^2 \cdot 5$ &  $3^2 \cdot 7 \cdot 11 $& 13 & 3 & $\overline{2080}$ & $\overline{11781}$ \\
& $2^2 \cdot 7$ &  $3^2 \cdot 5 \cdot 11 $& 19 & 25 & $\overline{12376}$ & $\overline{1485}$ \\
& $2^2 \cdot 11$ &  $3^2 \cdot 5 \cdot 7 $& 7 & 25 & $\overline{7876}$ & $\overline{5985}$ \\
& $3^2 \cdot 5$ &  $2^2 \cdot 7 \cdot 11 $& 38 & 13 & $\overline{4005}$ & $\overline{9856}$ \\
& $3^2 \cdot 7$ &  $2^2 \cdot 5 \cdot 11 $& 31 & 2 & $\overline{441}$ & $\overline{13420}$ \\
& $3^2 \cdot 11$ &  $2^2 \cdot 5 \cdot 7 $& 41 & 70 & $\overline{9801}$ & $\overline{4060}$ \\
& $ 5\cdot 7$ &  $2^2 \cdot 3^2 \cdot 11 $& 11 & 19 & $\overline{7525}$ & $\overline{6336}$ \\
& $ 5\cdot 11$ &  $2^2 \cdot 3^2 \cdot 7 $& 32 & 12 & $\overline{3025}$ & $\overline{10836}$ \\
& $ 7\cdot 11$ &  $2^2 \cdot 3^2 \cdot 5 $& 26 & 74 & $\overline{13321}$ & $\overline{540}$ \\
\hline
\end{tabular}
\end{center}
\begin{center}
    Table 5.1
\end{center}
\noindent Including $\overline{0}$ and $\overline{1}$, we obtain the $2^5=32$ idempotents of $\mathbbm{Z}_{13860}$ from the Table 5.1.\\
 \underline{\textbf{Working method (1)(b)}} :  Consider that $n=2^2\cdot3\cdot5\cdot7=420$. Here $k=4$. Now we calculate the elements of $1_+$. $2_+$ in the  Table 5.2 below, using the second part of Proposition 5.1.
\begin{center}
   \strutlongstacks{T}
\begin{tabular}{ |c|c|c|c|c|c|c| } 
\hline
$\Longstack{set\\of\\idempo-\\tents}$ & \Longstack{$p= \displaystyle \prod_{i=1} ^{h}p_{i}^{\alpha_i}$\\ \\$1 \leq h \leq 2$} & \Longstack{$m={\displaystyle \prod_{i=1}^{k} p_{i}^{\alpha_i}}$\\ \\$i \neq j_1,\cdots, j_h$\\$(m,p)=1$} &\Longstack{$m=ps+t,$ \\$s (\geq 0)\in \mathbbm{Z}$\\ $1 \leq t \leq p-1$} & \Longstack{ $1 \leq r \leq p-1, $ \\ $p \mid rt+1$ \\$(p,r)=1$ }& {idempotents}&idempotents\\
\hline
$h_+$ & $p$ & $m$ & $t$ & $r$ & $\overline{r \cdot m+1}$ & $\overline{(p-r)\cdot m}$\\
\hline
\multirow{4}{1em}{$1_+$} & $2^2$ & $3 \cdot 5 \cdot 7 $& 5 & 3 & $\overline{316}$ & $\overline{105}$\\
& $3$ & $ 2^2\cdot 5 \cdot 7 $& 2 & 1& $\overline{141}$ & $\overline{280}$\\
& $5$ & $ 2^2\cdot 3 \cdot 7 $& 4 & 1& $\overline{85}$ & $\overline{336}$\\
& $7$ & $ 2^2\cdot 3 \cdot 5  $& 4 & 5 & $\overline{301}$ & $\overline{120}$\\
\hline
\multirow{3}{1em}{$2_+$} & $2^2 \cdot 3$ & $ 5 \cdot 7 $& 11 & 1 & $\overline{36}$ & $\overline{385}$\\
& $2^2 \cdot 5 $ & $ 3 \cdot 7 $& 1 & 19 & $\overline{400}$ & $\overline{21}$\\
& $2^2 \cdot 7 $ & $ 3 \cdot  5  $& 15 & 13 & $\overline{196}$ & $\overline{22 5}$\\
\hline
\end{tabular}
\end{center}
\begin{center}
    Table 5.2
\end{center}
Including $\overline{0}$ and $\overline{1}$, we obtain $ 2^4=16 $ idempotents of $\mathbbm{Z}_{420}$ from the Table 5.2.\\
\begin{proposition}
    Let $n= p_1^{\alpha _1}p_2^{\alpha _2}\cdots p_k^{\alpha _k}$, where $k(\geq 2) \in \mathbbm{N}$, $1\leq i \leq k  $, $\alpha_i \in \mathbbm{N} $, $p_i$'s are distinct primes, then all distinct idempotents of $\mathbbm{Z}_n$ can be calculated by obtaining the elements of the form $\overline{r\cdot m+1}$ (using Theorem 3.1) in $h_+$, $\forall{h}$, where $1 \leq h \leq k-1$, $p$ and $m$ are as defined in $h_+$, $m=ps+t$ with $1 \leq r,t \leq p-1 $ and $s ( \geq 0) \in \mathbbm{Z}$ such that $p \mid rt+1$.
\end{proposition}
\begin{proof}
We consider the sets of idempotents $1_+,2_+,\cdots , (k-1)_+$. Let $1\leq h\leq k-1$, then we have the following cases.

\noindent \textbf{Case 1 }: Let $h\neq \frac{k}{2}$.\\ 
For  $1\leq h\leq k-1$, from Theorem 4.1, we  have that, for any two different representations of $n$ given by $n=pm$ and $n={p'}{m'}$ in $h_+$, the idempotents of $\mathbbm{Z}_n$ of the form $\overline{r\cdot m+1}$ and $\overline{{r'}\cdot {m'}+1}$, are different, i.e., $\overline{r\cdot m+1}\neq \overline{{r'}\cdot {m'}+1}$.\\
For $1\leq j \neq g \leq k-1$, let us consider the sets $j_+$ and $g_+$ and let $n=pm$ be a representation of $n$ in $j_+$ and $n={p'}{m'}$ be a different representation of $n$ in $g_+$. As we have $1\leq j \neq g \leq k-1$ and we are considering different representations of  $n$, trivially we have $p\neq {p'}$ and $m \neq {m'}$.\\
If $p={m'}$ and ${p'}=m$, and $\overline{{r'}\cdot {m'}+1}$, $\overline{r\cdot m+1}$ be the idempotents we obtain (using Theorem 3.1) for the representations $n=pm$ and  $n={p'}{m'}$ respectively, then using Lemma 4.1 and the first part of Theorem 4.1, we obtain that $\overline{r\cdot m+1}= \overline{({p'}-{r'})\cdot{m'}} \neq \overline{{r'}\cdot {m'}+1}$, i.e., $\overline{r\cdot m+1} \neq \overline{{r'}\cdot {m'}+1}$.\\
If $p\neq{m'}$ and ${p'}\neq m$ with $p\neq {p'}$ and $m \neq {m'}$, then similar as the proof of first part of Theorem 4.1, we obtain that $\overline{r\cdot m+1} \neq \overline{{r'}\cdot {m'}+1}$.

\noindent \textbf{Case 2} : Let $h= \frac{k}{2}$. Then we consider two different representations $n=pm$ and $n={p'}{m'}$ of $n$ in $h_+$. Similar to the previous case, we have that, $\overline{r\cdot m+1} \neq \overline{{r'}\cdot {m'}+1}$.\\
From these results, we conclude that, for any two different representations of $n$, the idempotents that we obtain using Theorem 3.1 of the form  $\overline{r\cdot m+1}$, are distinct. Now there are total ${k}\choose {h}$ number of different representations of $n$ in $h_+$, $\forall h$, where $1\leq h\leq k-1$. For  each representation of $n$, we obtain an idempotent of the form $\overline{r\cdot m+1}$ and there are total  ${{k} \choose {1}}+{{k} \choose {2}}+ \cdots +{{k} \choose {k-1}}=2^k-2$ number of distinct representations of $n$ and for each of the representation, we obtain an idempotent of $\mathbbm{Z}_n$, which is different from other idempotents, obtained in this way.\\
Hence, for total $2^k-2$ different  representations of $n$, we obtain $2^k-2$ distinct nontrivial idempotent of $\mathbbm{Z}_n$ of the form $\overline{r\cdot m+1}$, where $m,r$ are as defined in Theorem 3.1. Including the trivial idempotents, we have $2^k$ idempotents of $\mathbbm{Z}_n$.\\
Therefore, all $2^k$ distinct idempotents of $\mathbbm{Z}_n$ can be calculated by obtaining the elements of the form $\overline{r\cdot m+1}$ (using Theorem 3.1) in $h_+$, $\forall{h}$, where $1 \leq h \leq k-1$, $p$ and $m$ are as defined in $h_+$, $m=ps+t$ with $1 \leq r,s \leq p-1 $ and $s ( \geq 0) \in \mathbbm{Z}$ such that $p \mid rt+1$.    
\end{proof}
Proposition 5.2 can be stated in a different yet similar way if we consider the idempotents of $\mathbbm{Z}_n$ of the form $\overline{(p-r)\cdot m}$, instead of the idempotents of the form $\overline{r\cdot m+1}$.
\begin{proposition}
    Let $n= p_1^{\alpha _1}p_2^{\alpha _2}\cdots p_k^{\alpha _k}$, where $k(\geq 2) \in \mathbbm{N}$, $1\leq i \leq k  $, $\alpha_i \in \mathbbm{N} $, $p_i$'s are distinct primes, then all the distinct idempotents of $\mathbbm{Z}_n$ can be calculated by obtaining the elements of the form $\overline{(p-r)\cdot m}$ (using Theorem 3.1) in $h_+$, $\forall{h}$, where $1 \leq h \leq k-1$, $p$ and $m$ are as defined in $h_+$, $m=ps+t$ with $1 \leq r,s \leq p-1 $ and $s ( \geq 0) \in \mathbbm{Z}$ such that $p \mid rt+1$.
\end{proposition}
\begin{proof}
    This  result can be proved in a similar way as Proposition 5.2. 
\end{proof}
The working method for Proposition 5.2 is given below.\\
 \underline{\textbf{Working method (2)}} :  Consider that $n=2^2\cdot3\cdot5\cdot7=420$. From Proposition 5.2, we need to calculate the elements of $1_+,2_+,3_+$ of the form $\overline{r\cdot m+1}$ to obtain the idempotents of $\mathbbm{Z}_{420}$, which is given in the  Table 5.3 below.
 \begin{center}
   \strutlongstacks{T}
\begin{tabular}{ |c|c|c|c|c|c| } 
\hline
$\Longstack{set\\ of\\idempo-\\tents}$ & \Longstack{$p= \displaystyle \prod_{i=1} ^{h}p_{i}^{\alpha_i}$\\ \\$1 \leq h \leq 3$} & \Longstack{$m={\displaystyle \prod_{i=1}^{k} p_{i}^{\alpha_i}}$\\ \\$i \neq j_1,\cdots, j_h$,\\$(m,p)=1$} &\Longstack{$m=ps+t,$ \\$s (\geq 0)\in \mathbbm{Z}$,\\ $1 \leq t \leq p-1$} & \Longstack{ $1 \leq r \leq p-1, $ \\ $p \mid rt+1$, \\$(p,r)=1$ }& {idempotents}\\
\hline
$h_+$ & $p$ & $m$ & $t$ & $r$ & $ \overline{r \cdot m+1}$\\
\hline
\multirow{4}{1em}{$1_+$} & $2^2$ & $3 \cdot 5 \cdot 7  $ & 5 & 3 & $\overline{316}$\\ 
& $3$ & $2^2 \cdot 5 \cdot 7  $ & 2 & 1 & $\overline{141}$\\
& $5$ & $2^2 \cdot 3 \cdot 7  $ & 4 & 1 & $\overline{85}$\\
& $7$ & $2^2 \cdot 3 \cdot 5  $ & 4 & 5 & $\overline{301}$\\
\hline
\multirow{6}{1em}{$2_+$} & $2^2 \cdot 3$ & $ 5 \cdot 7  $ & 11 & 1 & $\overline{36}$\\
 & $2^2 \cdot 5$ & $ 3 \cdot 7  $ & 1 & 19 & $\overline{400}$\\
  & $2^2 \cdot 7$ & $ 5 \cdot 3  $ & 15 & 13 & $\overline{196}$\\
   & $3 \cdot 5$ & $ 2^2 \cdot 7  $ & 13 & 8 & $\overline{225}$\\
    & $3 \cdot 7$ & $ 2^2 \cdot 5  $ & 20 & 1 & $\overline{21}$\\
    & $5 \cdot 7$ & $ 2^2 \cdot 3  $ & 12 & 32 & $\overline{385}$\\
    \hline
 \multirow{4}{1em}{$3_+$} & $2^2 \cdot 3\cdot 5$ &  7   & 7 & 17 & $\overline{120}$\\ 
 & $2^2 \cdot 3\cdot 7$ &  5   & 5 & 67 & $\overline{336}$\\ 
& $3 \cdot 5\cdot 7$ &  $2^2$   & 4 & 26 & $\overline{105}$\\ 
& $2^2 \cdot 5\cdot 7$ &  3   & 3 & 93 & $\overline{280}$\\
\hline
\end{tabular}
\end{center}
\begin{center}
    Table 5.3
\end{center}
Including $\overline{0}$ and $\overline{1}$, we obtain the $2^4=16$ idempotents of $\mathbbm{Z}_{420}$ from the Table 5.3.

Similarly using Proposition 5.3, we can have  working method (3) to obtain $2^4=16$ idempotents of $\mathbbm{Z}_{420}$ from the Table 5.3, after  replacing $\overline{r\cdot m+1}$ by $\overline{(p-r)\cdot m}$.
\begin{theorem}
    Every nontrivial idempotent of $\mathbbm{Z}_{n}$ is either of the form $\overline{r\cdot m+1}$ or $\overline{(p-r)\cdot m}$ or both, for some positive integer $m,r,p$ with $n=pm$, $(p,m)=1$, $1\leq r \leq p-1$, $(p,r)=1$ and $p \mid rm+1$.  
\end{theorem}
\begin{proof}
    From Proposition 5.2, we get that, all the nontrivial idempotents of $\mathbbm{Z}_{n}$ can be calculated by obtaining the elements of the form $\overline{r\cdot m+1}$ of the set $h_+$ of idempotents. This implies that, all the nontrivial idempotents of $\mathbbm{Z}_{n}$ are of the form $\overline{r\cdot m+1}$, where $m,r,p$ are as defined in Proposition 5.2 or in Theorem 3.1. \\
    Similarly, using Proposition 5.3, we have that, all the nontrivial idempotents of $\mathbbm{Z}_{n}$ are of the form $\overline{(p-r)\cdot m}$, where $m,r,p$ are as defined in Proposition 5.3 or in Theorem 3.1.\\
    If we obtain the nontrivial idempotent of $\mathbbm{Z}_{n}$ using Proposition 5.1, then some idempotents are of the form $\overline{r\cdot m+1}$ and some idempotents are of the form $\overline{(p-r)\cdot m}$. Again for the same $n$, we can obtain the idempotents of $\mathbbm{Z}_{n}$ of the form $\overline{r\cdot m+1}$ (using Proposition 5.2) or of the form $\overline{(p-r)\cdot m}$ (using Proposition 5.3).\\
    So any nontrivial idempotent of $\mathbbm{Z}_{n}$ can be of the form $\overline{r\cdot m+1}$ or $\overline{(p-r)\cdot m}$ or both, for some positive integer $m,r,p$ satisfying the conditions in the hypothesis.
\end{proof}
The working methods (2) and (3) are almost similar, that's why we treat them in the same way, i.e., to obtain all idempotents of $\mathbbm{Z}_{n}$, we consider any one out of these 2 methods. However, the working methods (1) and (2) are bit different. There are some disadvantages of these 2 methods, compared to one another. For example, if we consider working method (1), then we need to point out the number of distinct prime divisor of $n$ and later we proceed according as the number is even or odd. There is no such procedure in working method (2).\\
However, we know that calculating $r$, using Theorem 3.1 and Theorem 3.3, is complex and working method (2) involves more calculation of $r$ than in working method (1).\\
So, both the working method (1) and (2) have some disadvantages compared to one another. Thus we can use any one out of these 3 different methods (using proposition 5.1, 5.2, 5.3) to obtain all idempotents of $\mathbbm{Z}_{n}$.
	\bibliographystyle{plain}
    \bibliography{ref1.bib}
	
\end{document}